\documentclass[cag,a4paper]{ipart_v2}

\Vol{34}
\Issue{1}
\Year{2026}
\Doi{10.4310/CAG.260505010647}
\firstpage{27}
\Pubonline{May 5, 2026}

\usepackage[english]{babel}
\usepackage[utf8]{inputenc}

\usepackage{enumerate}
\usepackage{amssymb}
\usepackage{amsmath}
\usepackage{mathrsfs}
\usepackage{amsthm}
\usepackage{mathtools}
\usepackage{tikz}

\theoremstyle{definition}
\newtheorem{theorem}{Theorem}[section]
\newtheorem{definition}[theorem]{Definition}

\newtheorem{lemma}[theorem]{Lemma}
\newtheorem{question}[theorem]{Question}

\newtheorem{corollary}[theorem]{Corollary}
\theoremstyle{remark}

\newtheorem{example}[theorem]{Example}
\numberwithin{equation}{section}

\begin{document}

\title[Remarks on minimal hypersurfaces]{Remarks on minimal hypersurfaces in shrinking gradient Ricci solitons}

\author[Y. Sun and G. Zhu]{Yukai Sun and Guangrui Zhu}

\dedicatory{Dedication to Hamilton}

\begin{abstract}
In this paper, we prove that any compact 2-sided smooth stable minimal hypersurface in a shrinking gradient Ricci soliton $(M^{n},g,f)$ with scalar curvature $R\geq(n-1)\lambda$ must have vanished second fundamental form and vanished normal Ricci curvature. For shrinking gradient Ricci solitons with scalar curvature $R\geq(n-1)\lambda$, the existence of an area-minimizing hypersurface would imply that $M$ is splitting.
\end{abstract}

\maketitle

	\section{Introduction}
	
	A smooth compact(without boundary) minimal hypersurface $\Sigma$ in a Riemannian manifold $(M,g)$ is stable if for any $\phi\in C^{\infty}(\Sigma)$ such that
    \[\int_{\Sigma}|\nabla^{\Sigma}\phi|^2-(\operatorname{Ric}(\nu,\nu)+|A|^2)\phi^2\geq 0,\]
    where $\nabla^{\Sigma}$ is the connection on $\Sigma$ induced from $(M,g)$, $\operatorname{Ric}$ is the Ricci curvature of $M$, $\nu$ is the outwards pointing unit normal vector of $\Sigma$ and $A$ is the fundamental form of $\Sigma$. 
    
    Fischer-Colbrie-Schoen \cite{FR1980}(see also Schoen-Yau \cite{SY1982}) obtained a classification for the stable minimal hypersurface in complete three-dimensional manifolds with nonnegative scalar curvature. Chodosh-Li-Stryker\cite{CLS2024} showed that the stable minimal hypersurface $\Sigma$ is totally geodesic and $\operatorname{Ric}(\nu,\nu)=0$ in four-dimensional manifolds with scalar curvature bounded from below by a positive constant, nonnegative sectional curvature, and weakly bounded geometry. Hong-Yan \cite{HY2024} also found that the stable minimal hypersurface $\Sigma$ is totally geodesic and $\operatorname{Ric}(\nu,\nu)=0$ in five-dimensional Riemannian manifolds under some curvature conditions. 
    
    We aim to study compact minimal hypersurfaces in gradient Ricci solitons with scalar curvature bounded from below. First, we recall the definition of gradient Ricci solitons.
	\begin{definition}
		A gradient Ricci soliton, denoted by $(M,g,f)$, is a smooth $n$-dimensional complete Riemannian manifold $(M,g)$ equipped with a smooth function $f:M\to\mathbb{R}$, satisfying
		\begin{equation}\label{eqn-1}
			\operatorname{Ric}+\operatorname{Hess}f=\lambda g
		\end{equation}
		for some $\lambda\in\mathbb{R}$,
		where $\operatorname{Ric}$ is the Ricci tensor and $\operatorname{Hess} f$ is the Hessian of $f$.

        If $\operatorname{Hess}f$ is not equal to zero on $M$, we call the gradient Ricci soliton $(M,g,f)$ nontrivial.
	\end{definition}
	
	When $\lambda>0$, $\lambda=0$, or $\lambda<0$, the gradient Ricci soliton $(M,g,f)$ is referred to as a shrinking, steady, or expanding gradient Ricci soliton, respectively. A gradient Ricci soliton can be viewed as a special structure on a Riemannian manifold that relates the Ricci curvature and metric by a function and is of great importance in the study of singularities in Ricci flow \cite{CB2023, HR1995}. Therefore, we can ask whether a compact stable minimal hypersurface has vanishing second fundamental form and vanishing normal Ricci curvature under the assumption of a lower bound on scalar curvature. Indeed, we can demonstrate the following theorem:
	\begin{theorem}\label{Thm-main}
		Any smooth compact 2-sided stable minimal hypersurface is totally geodesic with vanishing normal Ricci curvature in a nontrivial shrinking gradient Ricci soliton $(M^n, g, f)$ with scalar curvature $R\geq (n-1)\lambda$.
	\end{theorem}
    
    We can also obtain an analogous result for gradient steady and expanding Ricci solitons under a similar condition as in Theorem \ref{Thm-main}. However, for steady and expanding gradient Ricci solitons, Munteanu-Wang\cite{WM2023} have proved that: 
    \begin{theorem}[Theorem 3.1\cite{WM2023}]\label{thm-WM1}
        Let $(M^{n},g,f)$ be a steady Ricci soliton. Assume that there exists a smooth compact embedded minimal hypersurface $\Sigma$ in $M$. Then $(M,g)$ splits isometrically as a direct product $\Sigma\times \mathbb{R}$.
    \end{theorem}
        \begin{theorem}[Theorem 3.2\cite{WM2023}]\label{thm-WM2}
        Let $(M^{n},g,f)$ be an expanding Ricci soliton with scalar curvature $R\geq \lambda(n-1)$. Assume that there exists a smooth compact embedded minimal hypersurface $\Sigma$ in $M$. Then $(M,g)$ splits isometrically as a direct product $\Sigma\times \mathbb{R}$.
    \end{theorem}
	
Now we consider compact area-minimizing hypersurfaces. We recall the definition of an area-minimizing hypersurface.
\begin{definition}
    Let $\Sigma$ be a compact hypersurface in a Riemannian manifold $(M,g)$. Suppose that the homology class $[\Sigma]\in H_{n-1}(M,\mathbb{Z})$ of $\Sigma$ is non-zero, i.e., $[\Sigma]\neq 0$ in $H_{n-1}(M,\mathbb{Z})$. Then $\Sigma$ is called an area-minimizing hypersurface  if 
    \[
    \operatorname{Area}(\Sigma)=\inf_{\hat{\Sigma}\in [\Sigma]}\{\operatorname{Area}(\hat{\Sigma})\},
    \]
    where $\operatorname{Area}(\Sigma)$ denotes the area of $\Sigma$.
\end{definition}

If $n>7$, the area-minimizing hypersurfaces may not be smooth. However, if $n\leq 7$, all area-minimizing hypersurfaces are smooth \cite{SL1983}. Therefore, we consider the dimension of $M$ between $3$ and $7$. 

Bray-Brendle-Neves\cite{BBN2010}, Hao-Shi-Sun\cite{TYY2024}, and Zhu\cite{zhu2020} have shown that the existence of a compact area-minimizing hypersurface implies a local splitting property under some proper conditions. Motivated by their work, we have the following results for shrinking Ricci solitons.
	\begin{theorem}\label{thm-shrinking}
		For $3\leq n
		\leq 7$, let $(M^n, g, f)$ be a shrinking gradient Ricci soliton with $R\geq(n-1)\lambda$. If there exists a compact 2-sided area-minimizing hypersurface $\Sigma$ in $M$, then $\Sigma$ must be an Einstein manifold, and $M$ is isometric to the Riemannian product $\Sigma\times\mathbb{R}$.
	\end{theorem}
From Theorem \ref{thm-shrinking}, we have 
\begin{corollary}
    For $3\leq n
		\leq 7$, let $(M^n, g, f)$ be an orientable compact nontrivial shrinking gradient Ricci soliton with $R\geq(n-1)\lambda$. Then $H_{n-1}(M,\mathbb{Z})= 0$.
\end{corollary}
    
    %\begin{remark}
 %If $M^{n}$ is an oriented manifold with $H_{n-1}(M^{n},\mathbb{Z})\neq 0$ for $3\leq n\leq 7$, then there is a point $x\in M$ such that $R_{g}(x)<(n-1)\lambda$ by Theorem \ref{thm-shrinking}.
   % \end{remark}

    Theorem \ref{thm-shrinking} is related to the following question raised by Munteanu-Wang in \cite{WM2022JGA}.
    \begin{question}[Question 1.1 \cite{WM2022JGA}]
        If a complete shrinking gradient Ricci soliton $M$ contains a line, is it isometric to the product of $\mathbb{R}\times N$ for some~$N$?
    \end{question}
    Therefore, if $(M^{n}, g, f)$ is a complete noncompact nontrivial shrinking gradient Ricci soliton with at least two ends. And if there is a compact area-minimizing hypersurface $\Sigma$ which $[\Sigma]\in H_{n-1}(M,\mathbb{Z})$ is nontrivial and $R_{g}\geq (n-1)\lambda$ on $M$. Then, in this case, the above Question is true by Theorem \ref{thm-shrinking}.
    
    The proof of Theorem \ref{thm-shrinking} uses the $\mu$-bubble method as was done in \cite{zhu2020}. One key step for local splitting is to carefully choose the function $\mu$ in equation (\ref{eqn-function-h}). 

    The paper is organized as follows. In section \ref{sec-preli}, we give some facts about the gradient Ricci soliton and $\mu$-bubble. In section \ref{sec-proof}, we prove Theorem \ref{Thm-main} and Theorem \ref{thm-shrinking}.

    {\em Acknowledgements}: The authors wish to express their gratitude to the anonymous referees for providing suggestions and comments. Yukai Sun is partially funded by the National Key R\&D Program of China Grant 2020YFA0712800.

	\section{Basic fact about gradient Ricci soliton and \texorpdfstring{$\mu$}{mu}-bubble}\label{sec-preli}
	In this section, we revisit some key facts concerning gradient Ricci solitons and $\mu$-bubbles.
    
    The following lemma is well known for gradient Ricci solitons. For completeness, we give a proof.
	\begin{lemma}For gradient Ricci solitons $(M,g,f)$, We have the following equations
		\begin{align}
			R+\Delta f &= \lambda n,\label{eqn-Delta}\\
			R+|\nabla f|^2-2\lambda f &= 0\,(\lambda\neq 0) .\label{eqautionr+f1}
		\end{align}
	\end{lemma}
	\begin{proof}
		We compute in a local normal coordinate $\{\partial_{i}\}_{i=1}^{n}$.
		Taking trace for formula \eqref{eqn-1}, we get the formula (\ref{eqn-Delta}).\\
		Since $$\nabla_i\nabla_j\nabla_k f-\nabla_j\nabla_i\nabla_k f=-R_{ijkl}\nabla_l f,$$ and
		$$\nabla_j\nabla_k f=\lambda g_{jk}-R_{jk},\,\nabla_i\nabla_k f=\lambda g_{ik}-R_{ik},$$
		we have
		\begin{equation}\label{eqn-6}
			\nabla_j R_{ik}-\nabla_i R_{jk}=-R_{ijkl}\nabla_l f.
		\end{equation}
		Taking trace about $k,j$ for \eqref{eqn-6} and use the contracted second Bianchi identity $2\nabla_j R_{ij}=\nabla_{i}R$,
		we have $$\nabla_{j}R_{ij}-\nabla_{i} R=-R_{il}\nabla_l f$$ and
		$$\nabla_{i}R=2R_{il}\nabla_l f.$$
		We compute
		\begin{align*}
			\nabla_i(R+|\nabla f|^2-2\lambda f) &= \nabla_i R+2\nabla_i\nabla_{l}f\nabla_lf-2\lambda \nabla_{i} f \\
			&=2(R_{il}+\nabla_i\nabla_{l}f)\nabla_lf-2\lambda \nabla_{i} f  \\
			&= 2\lambda g_{il}\nabla_lf-2\lambda \nabla_{i} f\\
			&=0.
		\end{align*}
		So $R+|\nabla f|^2-2\lambda f=constant$. We can add another constant $C$ to $f$ to make the constant zero without changing the equation for $\lambda\neq 0$. Thus, we get the equation \eqref{eqautionr+f1}.
	\end{proof}
	Next, we recall the basic knowledge of $\mu$-bubble. A Riemannian band $(M^{n},g)$ is a compact connected orientable smooth manifold with a metric $g$ and nonempty boundary $\partial M$ such that  \[\partial M=\partial_{-}M\cup \partial_{+}M, \quad\partial_{-}M \neq \emptyset,\quad \partial_{+}M \neq \emptyset, \quad\partial_{-}M\cap \partial_{+}M=\emptyset.\] 
	On a Riemannian band $(M^{n},g)$, let $\overset{\circ}{M}$ denote the interior of $M$ and $h$ be a smooth function either on $\overset{\circ}{M}$ or on $M$. A Caccioppoli set $\bar{\Omega}\subset M$ is chosen such that its boundary $\partial \bar{\Omega}\subset \overset{\circ}{M}$ is smooth and $\partial_{-}M\subset \bar{\Omega}$. Let $\partial^{\ast} \Omega$ denote the reduced boundary of the Caccioppoli set $\Omega$. We consider the following functional 
	\begin{equation}\label{eqn-omega}
		E (\Omega) = \int_{\partial^{\ast} \Omega}\mathrm{d}\mathcal{H}^{n-1} - \int_{\Omega} (\chi_{\Omega}
		- \chi_{\bar{\Omega}})h\,\mathrm{d}\mathcal{H}^{n} 
	\end{equation}
	for $\Omega\in \mathcal{C}$, where $\mathcal{C}$ is defined as
	\[\mathcal{C}=\{\Omega|\mbox{ all Caccioppoli sets } \Omega \subset M \mbox{ and }\Omega\triangle \bar{\Omega}\Subset \overset{\circ}{M}\}, \]
	here $\mathcal{H}^{n}$ denotes $n$-dimensional Hausdorff measure.  We usually omit the measure $\mathcal{H}^{n}$ in the integral for simplicity. Then $\mu$-bubbles are critical points of the functional $E(\Omega)$.
	
	From Proposition 2.1 in \cite{zhu2021} and Section 5.1 in \cite{Gromov2023}, we have the following existence result of the $\mu$-bubble.
	\begin{lemma}[Existence of $\mu$-bubble]\label{lem-existence}
		For a Riemannian band $(M^{n},g)$ with $3\leq n\leq 7$, if either $h\in C^{\infty}(\overset{\circ}{M})$ with $h\to \pm\infty$ on $\partial_{\mp}M$, or $h\in C^{\infty}(M)$ with
		\[h|_{\partial_{-}M}>H_{\partial_{-}M},\quad h|_{\partial_{+}M}<H_{\partial_{+}M}\]
		where $H_{\partial_{-}M}$ is the mean curvature of $\partial_{-}M$ with respect to the inward normal and $H_{\partial_{+}M}$
		is the mean curvature of $\partial_{+}M$ with respect to the outward normal. Then there exists an $\Omega\in \mathcal{C}$ with smooth boundary such that
		\[E(\Omega)=\inf_{\Omega'\in\mathcal{C}}E(\Omega').\]
	\end{lemma}
	In the following, we discuss the first and second variations of $\mu$-bubbles.
	\begin{lemma}[First and second variation of $\mu$-bubbles]\label{lem-first-variation}
		Let $\Omega_{t}$ be a smooth $1$-parameter family of region in $\mathcal{C}$ with $\Omega_{0}=\Omega$ and normal speed $\phi$ at $t=0$, then
		\begin{equation}\label{eqn-first-variation}
			\frac{\mathrm{d}}{\mathrm{d} t} E (\Omega_t) |_{t = 0} =
			\int_{\partial \Omega} (H - h)  \phi ,
		\end{equation}
		where $H$ is the mean curvature of $\partial \Omega$. In particular, a $\mu$-bubble $\Omega$ satisfies
		\[H= h.\]
	For the $\mu$-bubble $\Omega$, we have
		\begin{equation}\label{weighted stability}
			\frac{\mathrm{d}^2}{\mathrm{d} t^2} E (\Omega_t) |_{t = 0}
			=  \int_{\partial \Omega} [- \Delta_{\Sigma} \phi - |A|^2 \phi
			-\operatorname{Ric}(\nu, \nu) \phi-\langle\nabla h,\nu\rangle \phi] \phi \geq 0.
		\end{equation}
		where $\nu$ is the outwards pointing unit normal vector on $\partial \Omega$.
	\end{lemma}
	\section{Proof of theorems}\label{sec-proof}
	In this section, we prove our results. 
	
	We first recall the properties of minimal hypersurfaces. If there is a compact minimal hypersurface $\Sigma$ in a gradient Ricci soliton $(M^{n},g,f)$. Then the mean curvature $H$ of $\Sigma$ is zero.  By the stable condition and the second variation of $\Sigma$, we have
	\begin{eqnarray}\label{equ-Second-variation}
		\int_{\Sigma}\left[|\nabla^{\Sigma} \phi|^2-\left(\operatorname{Ric}(\nu,\nu)+|A|^2\right)\phi^2\right]\geq 0,
	\end{eqnarray}
	where $\phi \in C^{\infty}(\Sigma)$, $\nu$ is the outwards pointing unit normal vector on $\Sigma$ and $A$ is the second fundamental of $\Sigma$. On $\Sigma$, we also have
	\begin{eqnarray}\label{equ-Hessf}
		\operatorname{Hess} f(\nu,\nu)=\Delta f-\Delta_{\Sigma}f-H\langle \nu,\nabla f\rangle,
	\end{eqnarray}
	where $\Delta_{\Sigma}$ is the Laplace operator on $\Sigma$.

	\begin{proof}[Proof of Theorem \ref{Thm-main}]
		Assume that there is a compact smooth stable minimal hypersurface $\Sigma$ in $M$. 
		Since 
		\[\operatorname{Ric}(\nu,\nu)+\operatorname{Hess}f (\nu,\nu)=\lambda,\]
		by \eqref{equ-Second-variation}, we have
		\begin{eqnarray}
			\int_{\Sigma}\left[|\nabla^{\Sigma} \phi|^2-\left(\lambda-\operatorname{Hess}f(\nu,\nu)+|A|^2\right)\phi^2\right]\geq 0.
		\end{eqnarray}
		Combining with equations \eqref{equ-Hessf} and \eqref{eqn-Delta}, we have
		\begin{eqnarray}\label{eqn-simple}
			\int_{\Sigma}\left[|\nabla^{\Sigma} \phi|^2+\left(\lambda(n-1)-R-|A|^2-\Delta_{\Sigma}f\right)\phi^2\right]\geq 0.
		\end{eqnarray}
		Taking $\phi\equiv 1$, we obtain
		\begin{eqnarray}\label{eqn-phi=1}
			\int_{\Sigma}\left(\lambda(n-1)-R-|A|^2\right)\geq 0.
		\end{eqnarray}
		Therefore, if $R\geq\lambda(n-1)$, we have $A\equiv 0,$ $R=(n-1)\lambda$ and $\Delta f=\lambda$ on $\Sigma.$ 
	Hence,
	\[\int_{\Sigma}\operatorname{Ric}(\nu,\nu)=\int_{\Sigma}\left(\lambda-\operatorname{Hess}f(\nu,\nu)\right)=0.\]
	By equation \eqref{equ-Second-variation},
	\[\int_{\Sigma}\left[|\nabla^{\Sigma} \phi|^2-\operatorname{Ric}(\nu,\nu)\phi^2\right]\geq 0.\]
	Since the operator 
	$-\Delta_{\Sigma}-\operatorname{Ric}(\nu,\nu)$
	is nonnegative, its first eigenvalue $\lambda_{1}$ is nonnegative. By Rayleigh's principle,
	\[\lambda_{1}=\inf_{0\neq \phi\in H^{2}_{1}(\Sigma)}\frac{\int_{\Sigma}\left[|\nabla^{\Sigma}\phi|^2-\operatorname{Ric}(\nu,\nu)\phi^2\right]}{\int_{\Sigma}\phi^2}\leq \frac{\int_{\Sigma}\left[-\operatorname{Ric}(\nu,\nu)\right]}{\int_{\Sigma}}=0.\]
	Thus $\lambda_{1}=0$, i.e.
	\[-\Delta_{\Sigma}1-\operatorname{Ric}(\nu,\nu)1=0.\]
	Consequently, $\operatorname{Ric}(\nu,\nu)=0$, $\operatorname{Hess}f(\nu,\nu)=\lambda$ and $\Delta_{\Sigma}f=0$ on $\Sigma$. We get $f$ is a constant on $\Sigma$. 
    \end{proof}
    
   % From the above proof, we have the following results:
	%\begin{proposition}\label{prop-TG}
		%Let $(M^n,g,f)$ be a complete gradient Ricci soliton with $R\geq(n-1)\lambda$. If there exists a compact 2-side smooth stable minimal hypersurface $\Sigma$ in $M$, then $R|_{\Sigma}=(n-1)\lambda$, $f|_{\Sigma}=\mathrm{constant}$ and $\Sigma$ has vanishing second fundamental form and vanishing normal Ricci curvature.
	%\end{proposition}
	
	Before proof of Theorem \ref{thm-shrinking}, we need the following lemma, which was proved in Proposition 3.2\cite{BBN2010} or Lemma 3.3\cite{zhu2020}.
	\begin{lemma}\label{lem-CMC}
		Let $\Sigma$ be an area minimizing hypersurface in $M$ with vanished second fundamental form and vanished normal Ricci curvature, then we can construct a local foliation $\{\Sigma_{t}\}_{-\epsilon \leq t\leq \epsilon}$ in $M$ such that $\Sigma_{t}$ are of constant mean curvature and $\Sigma_{0}=\Sigma$.
	\end{lemma}
	
	\begin{proof}[Proof of Theorem \ref{thm-shrinking}]
		By Theorem \ref{Thm-main} and Lemma \ref{lem-CMC}, we can construct a local foliation $\{\Sigma_{t}\}_{-\epsilon \leq t\leq \epsilon}$ in $M$ such that $\Sigma_{t}$ are of constant mean curvature and $\Sigma_{0}=\Sigma$, let $H_t$ denote the mean curvature of $\Sigma_t$. We claim that $H_t\equiv0$ for any $t\in(-\epsilon,\epsilon)$. It suffices to prove this for nonnegative $t$, as the argument for the negative side is analogous. For contradiction, suppose that the assertion is not true. then there exists a $t_0>0$ such that $H_{t_0}=\delta>0$, since otherwise $H_{t}\leq0$ and the area-minimizing property of $\Sigma$ will imply  $H_{t}\equiv0$. Let $\Omega_{t_{0}}$ denote the region enclosed by $\Sigma$ and $\Sigma_{t_{0}}$ i.e.,
		\[\Omega_{t_{0}}:=\cup_{t\in [0,t_{0}]}\Sigma_{t}.\]
		Consider the following functional 
		\begin{equation}\label{eqn-function-h}
			E (\Omega) = \int_{\partial^{\ast} \Omega\setminus\Sigma}\mathrm{d}\mathcal{H}^{n-1} - \int_{\Omega} h(\chi_{\Omega}- \chi_{\bar{\Omega}})\mathrm{d} \mathcal{H}^{n},\; h=\delta e^{-f},
		\end{equation}
		for $\Omega,\bar{\Omega}\in \mathcal{C}$ and $\Sigma\subset \bar{\Omega}$, where 
		\[\mathcal{C}=\{\Omega|\mbox{ all Caccioppoli sets } \Omega \subset \Omega_{t_{0}} \mbox{ and }\Omega\triangle \bar{\Omega}\Subset \overset{\circ}{\Omega}_{t_{0}}\}. \]
		From equation \eqref{eqautionr+f1}, we have $0<h<\delta$, by Lemma \ref{lem-existence}, there exists a $\mu$-bubble $\hat{\Omega}$ such that $\hat{\Sigma}=\partial \hat{\Omega}\setminus\Sigma$ is a smooth two sided hypersurface disjoint from $\Sigma$ and $\Sigma_{t_{0}}$. Let $\hat{H}$ and $\hat{A}$ denote the mean curvature and the second fundamental form of the $\hat{\Sigma}$, respectively. Then we have $\hat{H}=h$ and 
		\[\int_{\hat{\Sigma}}\left[|\nabla^{\hat{\Sigma}} \phi|^2-\left(\operatorname{Ric}(\hat{\nu},\hat{\nu})+|\hat{A}|^2+\langle\nabla h,\hat{\nu}\rangle \right)\phi^2\right]\geq 0,\]
		where $\hat{\nu}$ is the outwards pointing unit normal vector on $\hat{\Sigma}$. Taking $\phi=1$ and combining with equations \eqref{eqn-Delta} and \eqref{equ-Hessf}, we obtain
		\begin{eqnarray}
			\nonumber
			\int_{\hat{\Sigma}}\left[(n-1)\lambda-R-h\langle \nabla f,\hat{\nu}\rangle-\langle\nabla h ,\hat{\nu}\rangle-|\hat{A}|^2\right]\geq 0.
		\end{eqnarray}
	We also have
	\begin{eqnarray}
		 %\nonumber % Remove numbering (before each equation)
		\int_{\hat{\Sigma}}\left(\lambda(n-1)-R-|\hat{A}|^2\right)\geq 0.
	\end{eqnarray} 
	Hence, $|\hat{A}|\equiv0$, which leads to a contradiction.
	
	Since $H_t$ vanishes for any $t\in(-\epsilon,\epsilon),$ there holds $\operatorname{Area}(\Sigma_t)=\operatorname{Area}(\Sigma)$, which yields
	that $\Sigma_t$ is also area-minimizing, by Theorem \ref{Thm-main}, $\Sigma_{t}$ is totally geodesic and has vanished normal Ricci curvature.
	
	In the following, we establish the existence of a local isometry $\Phi: \Sigma\times\mathbb{R}\to M.$
	Let $\tilde{\nu_t}$ and $\tilde{V_t}$  denote the outwards pointing unit normal vector field and normal variation vector field of the foliation $\{\Sigma_{t}\}_{-\epsilon\leq t\leq\epsilon}$, respectively. Assume $\Phi:\Sigma\times(-\epsilon,\epsilon)\to M$ is the flow generated by $\tilde{V_t}$. It follows that $\Phi$ is an embedding in a small neighborhood of $\Sigma$ and the pull-back of the metric $g$ is
	\[\Phi^{\ast}(g)=\psi^{2}\mathrm{d}t^{2}+\Phi^{\ast}_t(g_t),\]
	where $\psi=\langle\tilde{V_t},\tilde{\nu_t}\rangle>0$ is the lapse function (see definition in (3.3) of \cite{BBN2010}) and $g_t$ is the induced	metric on $\Sigma_{t}$ from $g$. Since $\Sigma_{t}$ is totally geodesic, we have $\partial_t\Phi^{\ast}_t(g_t)=2\psi  A_{\Sigma_t}=0$, which implies that $\Phi^{\ast}_t(g_t)=g|_{\Sigma}$. Additionally, the stability condition yields $-\Delta_{\Sigma_t} \psi - |A_{\Sigma_{t}}|^2 \psi
	-\operatorname{Ric}(\nu_t,\nu_t)\psi=-\Delta_{\Sigma_t} \psi=0,$ then $\psi(\cdot,t)\equiv\mathrm{constant}$. Let
	\[r(\cdot,t)=\int_{0}^{t}\psi(\cdot,s)ds,\]
	then \[\Phi^{\ast}(g)=\mathrm{d}r^2+g|_{\Sigma}.\]
	Hence, $\Phi: \Sigma\times(-\epsilon,\epsilon)\to M$ is local a isometry. Through a continuous argument as Proposition 3.8 in \cite{BBN2010}, we conclude that there exists a local isometry $\Phi: \Sigma\times\mathbb{R}\to M$. Moreover, $\Phi$ is a Riemannian covering map. Since $M$ has constant scalar curvature $R\equiv(n-1)\lambda,$ by Theorem 10 in \cite{MR}, $M$ is isometric to a product of an Einstein manifold with $\mathbb{R}^k$. We conclude $\Sigma$ is an Einstein manifold and $M$ is isometric to $\Sigma\times\mathbb{R}.$
	\end{proof}	
    We now give an example of the gradient Ricci soliton, which is the rigid case in Theorem \ref{thm-shrinking} 
\begin{example}
 $(\mathbb{S}^{n-1}\times \mathbb{R},g,f)$ with $g=g_{\mathbb{S}_{n-1}}+\mathrm{d}t^2$, $f(x,t)=\frac{\lambda t^2+(n-1)}{2}$ and $\lambda=n-2$.
\end{example}
At the end of this paper, we would like to ask the following question:
\begin{question}
	For $3\leq n
	\leq 7$ and $2\leq k\leq(n-2),$ let $(M^n, g, f)$ be a gradient Ricci soliton with $R\geq(n-k)\lambda$. If there exists a compact 2-sided area-minimizing surface $\Sigma$ of codimension $k$ in $M$, do we have $\Sigma$ must be an Einstein manifold and $M$ is isometric to the Riemannian product $\Sigma\times \mathbb{R}^k$\,?
	\end{question}

\bibliography{2-Sun}

\clearpage

\address{School of Mathematics and Statistics, Henan University\\
Kaifeng, Henan 475004, China\\
and Center for Applied Mathematics of Henan Province\\
Henan University, Zhengzhou, Henan 450046, China\\
\email{10100180@henu.edu.cn}}

\address{School of Mathematical Sciences\\
East China Normal University\\
Shanghai 200241, China\\
\email{52275500052@stu.ecnu.edu.cn}\\
\received{April 8, 2025}\\
\accepted{August 5, 2025}}

\end{document}